\documentclass[amsthm]{elsart}

\usepackage{yjsco}
\usepackage{natbib}
\usepackage{amsmath,amsthm,amssymb,graphicx}
\usepackage{savesym}
\savesymbol{AND}
\usepackage{algorithmic}
\usepackage{enumitem}
\usepackage{url}
\usepackage{epigraph}

\theoremstyle{definition}

\newtheorem{lemma}{Lemma}[section]
\newtheorem{proposition}[lemma]{Proposition}
\newtheorem{theorem}[lemma]{Theorem}
\newtheorem{corollary}[lemma]{Corollary}
\newtheorem{definition}[lemma]{Definition}
\newtheorem{example}[lemma]{Example}
\newtheorem{remark}[lemma]{Remark}

\newtheorem{algorithm}[lemma]{Algorithm}

\DeclareMathOperator{\lm}{lm}

\DeclareMathOperator{\lpos}{lpos}
\DeclareMathOperator{\POT}{POT}
\DeclareMathOperator{\Mon}{Mon}

\DeclareMathOperator{\Diag}{Diag}

\DeclareMathOperator{\Hom}{Hom}
\DeclareMathOperator{\Char}{char}

\DeclareMathOperator{\id}{id}

\DeclareMathOperator{\coker}{coker}

\DeclareMathOperator{\Ann}{Ann}

\DeclareMathOperator{\Quot}{Quot}

\newcommand{\N}{{\mathbb N}}
\newcommand{\R}{{\mathbb R}}
\newcommand{\K}{K}
\newcommand{\Z}{{\mathbb Z}}

\newcommand{\C}{{\mathbb C}}
\newcommand{\Q}{{\mathbb Q}}
\newcommand{\kg}{{\mathcal G}}
\newcommand{\kf}{{\mathcal F}}

\newcommand{\G}{{\mathcal G}}

\newcommand{\comment}[1]{}

\begin{document}
 \setcounter{section}{0}

\begin{frontmatter}

\title{Fraction-free algorithm for the computation of diagonal forms matrices over Ore domains using Gr{\"o}bner bases}

\author{Viktor Levandovskyy}\footnote{Corresponding author. Tel. +49 241 80 94546, Fax +49 241 80 92108}
\ead{Viktor.Levandovskyy@math.rwth-aachen.de}
\address{Lehrstuhl D f\"ur Mathematik, RWTH Aachen, Templergraben 64, 52062 Aachen, Germany}
\author{Kristina Schindelar}
\ead{kristina.schindelar@siemens.com}
\address{Siemens AG, Corporate Technology, Corporate Research and Technologies, Otto-Hahn-Ring 6,
81739 M\"unchen, Germany}


\begin{abstract}
This paper is a sequel to 
``Computing diagonal form and Jacobson normal form of a matrix using Gr\"obner bases'' \citep{LS10}.
We present a new fraction-free algorithm for the computation of a diagonal form of a matrix over a certain non-commutative Euclidean domain over a computable field with the help of Gr\"obner bases. 
This algorithm is formulated in a general constructive framework of non-commutative Ore localizations of $G$-algebras (OLGAs). 
We use the splitting of the computation of a normal form 
for matrices over Ore localizations into the diagonalization and the normalization processes. Both of them can be made fraction-free. 
For a given matrix $M$ over an OLGA $R$ we provide a diagonalization algorithm to compute $U,V$ and $D$ with fraction-free entries such that $UMV=D$ holds and $D$ is diagonal. The fraction-free approach
allows to obtain more information on the associated 
system of 
linear functional equations and its solutions, than the classical setup of an operator algebra with coefficients in rational functions.
In particular, one can handle distributional solutions together with, say, meromorphic ones. We investigate Ore localizations of common operator algebras over $\K[x]$ 
and use them in the unimodularity analysis of transformation matrices $U,V$.
In turn, this allows to lift the isomorphism of modules over an OLGA Euclidean domain to a smaller polynomial subring of it. We discuss the relation of this lifting with the solutions of the original system of equations.
Moreover, we prove some new results concerning normal forms of matrices over non-simple domains. Our implementation in the computer algebra system \textsc{Singular:Plural} follows the fraction-free strategy and shows impressive performance, compared with methods which directly use fractions. 
In particular, we experience moderate swell of coefficients and
obtain simple transformation matrices. Thus the method we propose is well suited for solving nontrivial practical problems.
\end{abstract}
\end{frontmatter}
\tableofcontents
\epigraph{We know what we are, but know not what we may be.}%
{William Shakespeare, ``Hamlet''.}

\section{Introduction}
This paper is a natural extension of our paper \citep{LS10}. We focus
on the algorithmic computation of a diagonal form of a given matrix
with entries in a Euclidean Ore algebra $A$. In contrast to commutative
principal ideal domains ($\Z$ or $\K[x]$ for a field $\K$), there is
no such strong normal form as the Smith normal form. However, in the
case when $A$ is a simple algebra, there is a celebrated Jacobson normal form \citep{Jacobson, Cohn} of remarkable structure. The existing proofs of diagonal and Jacobson normal forms offer an algorithm directly. However, such direct algorithms are rarely efficient in general.

We proposed the splitting of the computation of a normal form (like Jacobson form over simple domain) for matrices over Ore localizations into the \textit{diagonalization} and the \textit{normalization} processes. The diagonalization process can be formulated and applied in a fairly general setting, while the normalization depends on the algebra and on the properties of concretely given diagonal elements. 
Both diagonalization and normalization can be performed in a fraction-free fashion, and of course, the same ideas apply to the computation of the Smith normal form over a commutative Euclidean domain. As we demonstrate in 
Section \ref{NFGeneral},
the question of normal forms over non-simple domains is widely open and distinctly differs from the case where Jacobson form exists.


In this paper we concentrate on the diagonalization and present the fraction-free Algorithm \ref{diagonalPoly}, which is based on Gr{\"o}bner bases, and its implementation in \textsc{Singular:Plural}  \cite{Plural}. As a byproduct, we obtain a fraction-free algorithm for the computation of the Smith normal form over a commutative Euclidean domain.

Many operator algebras are non-commutative skew polynomial rings \citep{Kredel,ChyzakSalvy,BGV,Viktor,CQR}. We propose a new class of univariate skew polynomial rings, which are obtained as Ore localizations of $G$-algebras (OLGAs, see Def. \ref{OLGA}). This framework is powerful, convenient and constructive at the same time. 
Notably, it is possible to perform algorithmic computations like Gr\"obner bases for modules in such algebras. Moreover, many computations can be done in a fraction-free manner and hence they can be implemented in any computer algebra system, which can handle $G$-algebras or general polynomial Ore algebras.

Such algebras allow to describe time varying systems in systems and control theory \citep{Eva}, \citep{IlchmannMehrmann}, \citep{Ilchmann} and constitute a general framework for dealing with linear operator equations with variable coefficients. In \citep{CulianezQuadrat}, applications of a Jacobson form to systems of partial differential equations are shown and several concrete examples are introduced. 

In Section \ref{Solving} we discuss solutions to a given system of operator equations with variable coefficients. Starting with a polynomial ring $R_*$ and
its 'big' localization $R$, we denote by $M$ the presentation matrix $R_*$-module, corresponding to
the system of equations. We compute matrices $U,V$ and a diagonal matrix $D$, such that $UMV=D$ . 
Then, we determine a smaller localization $R'\subset R$ of $R_*$, over which both $U,V$ become invertible. We provide a module isomorphism for $R'$-modules, which allows to treat more general solutions, than those one can obtain with the $R$-module structure.

In Prop. \ref{cyclicAlg} we propose an algorithm for a probabilistic computation of a Jacobson form via a cyclic vector. Moreover, we provide both experimental data from computations and analysis of this data.

At the end, we pose five open problems. In the Appendix, we put the detailed explanation on the usage of our implementation.

\section{OLGAs and their Properties}
\label{Algebras}

By $\K$ we denote a computable field. Describing associative $K$-algebras via finite sets of generators $G$ and relations $R$, one usually writes $A = K\langle G \mid R \rangle = K\langle G \rangle / \langle R \rangle$. 
It means that $A$ is a factor algebra of the free associative algebra $K\langle G \rangle$ modulo the two-sided ideal, generated by $R$. Recall the following definition.

\begin{definition}
\label{Galg}
Let 
 $A$ be a quotient of the free associative algebra 
$\K\langle x_1,\ldots,x_n\rangle$ by the two-sided ideal $I$, generated by
the finite set $\{x_jx_i - c_{ij} x_ix_j-d_{ij} \} $ for all $ 1\leq i<j \leq n$, where
$c_{ij}\in \K^*$ and $d_{ij}$ are polynomials\footnote{Without loss of generality (cf. \citep{Viktor}) 
we can assume that $d_{ij}$ are given in terms of standard monomials $x_1^{a_1} \ldots x_n^{a_n}$} in $x_1,\ldots,x_n$.
Then $A$ is called a {\em $G$--algebra} \citep{LS03,Viktor}, if \\
$\bullet$ for all $ \; 1\leq i < j < k \leq n$ the expression
$c_{ik}c_{jk} \cdot d_{ij}x_k - x_k d_{ij} + c_{jk} \cdot x_j d_{ik} - c_{ij} \cdot d_{ik} x_j + d_{jk}x_i 
- c_{ij}c_{ik} \cdot x_i d_{jk}$
reduces to zero modulo $I$ and \\
$\bullet$ there exists a monomial ordering $\prec$ on $\K[ x_1,\ldots,x_n]$,
such that for each $i < j$ with $d_{ij}\not=0$, $\lm(d_{ij}) \prec x_i x_j$ . Here, $\lm$ stands for the classical notion of leading monomial of a polynomial from $\K[ x_1,\ldots,x_n]$.
\end{definition}
A monomial ordering on $K[x_1,\ldots,x_n]$ carries over to a $G$-algebra (and is called \textbf{admissible}) as in the definition as soon as it satisfies the second condition of the definition.

As in \cite{LS10},
we continue working with Ore localizations of $G$-algebras, which are principal ideal domains. A $G$-algebra $R$ is a Noetherian integral domain, hence there exists its total two-sided ring of fractions $\Quot(R) = (R\setminus\{0\})^{-1}R$, which is a division ring (skew field). It is convenient to define a class
of subalgebras between $R$ and $\Quot(R)$ as follows.

\begin{definition}
\label{OLGA}
Let $R$ be a $G$-algebra, generated by the set $X=\{x_1,\ldots,x_{n+1} \}$. Suppose, that
there is a $m$ with $1\leq m \leq n$ and a subset $Y = \{x_{i_1},\ldots,x_{i_m}\}$, such that $d_{i_j,i_k}$ does not involve other variables than those from $Y$. Then $B := \K\langle Y \mid \{x_jx_i - c_{ij} x_ix_j-d_{ij} \mid i<j, x_i,x_j \in Y \} \rangle$ is a $G$-algebra. For a proper completely prime ideal $I\subset B$,
$B/I$ is a domain. If, in addition, $S:=(B/I)\setminus\{0\}$\footnote{Note, that $I$ is an ideal in $B$ and may not be an ideal in $R$. Moreover, elements of $S$ are identified with residue classes modulo $I$.} is an Ore set in $A$, then we call $S^{-1}R$ an OLGA (Ore-localized $G$-algebra). It is encoded via the triple
$(R, B, I)$.
\end{definition}

From now on (for simplicity) we consider only OLGAs of the form $(R,B,0)$.

Prop. 28 of \citep{GML} implies the following characterization of OLGAs.
Let $Z = X \setminus Y$ in the notations as in the above definition.
If there exists an admissible $(Y,Z)$-block ordering on $R$, then $B\setminus\{0\}$ is an Ore set in $R$. 

If we put $m=n$ in Definition \ref{OLGA}, a corresponding OLGA is a Euclidean domain, which we then shorten as OLGAED. Let $X \setminus Y = \{ \d \}$.
In \cite{LS10}, we proved in Theorem 2.6, that for the case of OLGAED it is enough to require the existence of an admissible ordering on $A$, which satisfies $\d > x_{i_j}$ for all $1\leq j \leq n$. 
Moreover, the OLGAED $(B\setminus\{0\})^{-1}R$ can be presented as an Ore extension of $\Quot(B)$ by the variable $\d$.

\subsection{Notations} 
\label{orderingPOT}

In what follows we will work with OLGAED, given as Ore extension $R = A[\d; \sigma, \delta]$, where $A=\Quot(A_*)$ and $A_*$ is a $G$-algebra in variables $\{x_1,\ldots,x_n\}$. 
The computations will be
performed in a $G$-algebra $R_* = A_*[\d; \sigma, \delta]$\footnote{Note, that $A_* \subset A \subset R$ and $A_* \subset R_* \subset R$.}
with respect to the monomial module ordering $\POT$ (position-over-term), defined as follows. For $r, s \in \Mon(R_*)$ and the canonical basis $\{e_k\}$ of a free module of finite rank,
\begin{align*}
& r e_i < s e_j \;\; \Leftrightarrow \;\; i<j \mbox{ OR } (i=j \mbox{ and } r < s),
\end{align*}
and $r<s$ with respect to an admissible well-ordering on $R_*$, in which $\d$ is bigger
than any monomial not involving $\d$.
In $R$, a Gr\"obner basis is computed with respect to the induced POT ordering.

\section{Fraction-free or Polynomial Strategy}
\label{SecPoly}

Suppose a matrix $M$ over a non-commutative Euclidean domain $R$ is given. 
Without loss of generality, we suppose that $M$ does not contain a zero row.
In this section, we show our main approach of this paper. We introduce a method
that allows to execute Algorithm 3.5 from \cite{LS10} in a completely fraction-free framework. The idea comes from commutative algebra (see e.~g. \citep{GTZ}).
Gr\"obner bases were used for the computation of commutative Smith forms in e.~g. \citep{Spanier}.

We define the \textbf{degree} of an element in $R_*$ 
to be the weighted degree function with
weight $0$ to any generator of
$A_*$ and weight $1$ to $\partial$. Thus this weighted degree of $f\in R_*$
coincides with the degree of $f$ in $R$ and it is invariant under the multiplication by nonzero elements in $A_*$.

\begin{lemma}\label{inMod}
Let $M \in R^{p \times q}$. 
Then there exists a diagonal $R$-unimodular matrix $T\in A_*^{p \times p }$ 
such that $TM \in {R_{*}}^{p \times q}$. Moreover, the computation of such $T$ is algorithmic.
\end{lemma}
\begin{proof}
If $M \in {R_{*}}^{p \times q }$, there is nothing to do. Suppose that $M$ contains
elements with fractions.
At first, we show how to bring two fractional elements $a^{-1}b, c^{-1}d$ for
$a,c \in A_*$, $b,d \in R_*$ to a common left denominator, cf. \citep{ApelDiss}.
For any $h_1,h_2 \in A_*$, such that $h_1 a = h_2 c$, it is easy to see that
\[
(h_1 a)^{-1} (h_1 b) = a^{-1} h_1^{-1} h_1 b = a^{-1}b \text{ and }
(h_1 a)^{-1} (h_2 d) =  (h_2 c)^{-1}(h_2 d) = c^{-1}d,
\]
hence $(h_1 a)^{-1} = a^{-1} h_1^{-1} = (h_2 c)^{-1}$ is a common left denominator.
Analogously we can compute a common left denominator for any finite
set of fractions.
Let $T_{ii}$ be a common left denominator of all non-zero elements in the 
$i$-th row of $M$, then $TM$ contains no fractions. Moreover, $T$ is a diagonal matrix with non-zero entries from $A_*$, hence it is $R$-unimodular.
\end{proof}

\begin{remark}
Note that the computation of compatible factors $h_i$ for $a_1,a_2 \in A_*$
can be achieved by computing syzygies, since 
$\{(h_1,h_2)\in A_*^2 \mid h_1 a_1 = h_2 a_2 \}$ is precisely
the module $Syz(a_1,-a_2)\subset A_*^2$. The factors $h_i$ for more $a_i$'s can be obtained as well.
\end{remark}

\noindent
\textbf{Notation}. By $\mathcal{G}({}_{R_*}M_*)$ we denote the reduced left Gr{\"o}bner basis of the submodule ${}_{R_*}M_*$ with respect to the module ordering $<_*$
on $R_*$, defined in (\ref{orderingPOT}). 
Note, that monomials of $R_*^{1 \times q}$ are of the form $x_1^{\alpha_1} \cdots x_n^{\alpha_n}\partial^\beta e_k$ for $\alpha_i, \beta \in \N$, $1\leq k \leq q$ and $e_k$ is the $k$-th canonical basis vector.
Let $m$ be a nonzero vector with entries in $R_*$. Then by $\lm(m)$ we denote the leading monomial of $m$ with respect to $<_*$ and by $\lpos(m)=k$ the leading position of $\lm(m)$. By $\deg(m)$ we denote $\deg(\lm(m))=$ $\deg(x_1^{\alpha_1} \cdots x_n^{\alpha_n}\partial^\beta e_k) = \beta$.

For $M\in R^{p \times q}$, we denote by ${}_R M = R^{1\times p} M$ the left
$R$-module, generated by the rows of $M$.

Define $M_*:=T M \in R_*^{p \times q}$ using the notation of Lemma~\ref{inMod}.
Then the relations ${}_{R_*}M_* \subseteq {}_{R}M$ and 
${}_{R}M_*={}_{R}M$ hold obviously.
Thus whenever we speak about a finitely generated 
submodule ${}_{R}M \subset R^{1 \times q}$, we denote by ${}_{R}M_*$
a presentation of ${}_{R}M$ with generators contained in $R_*$.
In what follows, we will show how to find $R$-unimodular matrices $U \in R_*^{p \times p}$
and $V \in R_*^{q \times q}$ such that
\begin{footnotesize}
\begin{displaymath}
 U (TM) V = \left[ \begin{array}{ccc} r_1 & & \\  & \ddots &  \\ 
& & r_q \\ & 0 &   \end{array} \right] \in R_*^{p \times q}.
\end{displaymath}
\end{footnotesize}

\noindent
Since $U(TM)V = (UT)MV$
and $UT$ is a $R$-unimodular
matrix, our initial aim follows.\\

\begin{remark}\label{PolyIsBetter}
Using the fraction-free strategy, two improvements can be observed.
On the one hand, once we have mapped the matrix we work with from $R^{p \times q}$ to $R_*^{p \times q}$, the complicated arithmetics in the skew field of fractions is not used anymore. 
The other improvement lies in the nature of the construction of normal forms
for matrices and the corresponding transformation matrices.
The naive approach would be to apply elementary operations inclusive division by invertibles
on the rows and columns, that is, operations from the left and from the right. 
Indeed, there are methods using different techniques like, for instance, $p$-adic arguments to 
calculate the invariant factors of the Smith form over $\Z$ \citep{Frank}, but
this method does not help in constructing transformation matrices.
Surely the swap from left to right has no influence in the commutative framework.
But already in the rational Weyl algebra $B_1$, 
$\frac{1}{x}$ is an unit in $B_1$ and 
$\begin{displaystyle} \partial \tfrac{1}{x} =  \tfrac{1}{x} \partial - \tfrac{1}{x^2} \end{displaystyle}$.
Comparing the multiplication by the inverse element, that is, with $x$, we see that 
$\partial x=  x \partial + 1$ holds.
Thus a multiplication of any polynomial containing $\d$ with the element $\frac{1}{x}$ in the field of fractions  
causes an immediate coefficient
swell. Since a normal form of a matrix is given modulo unimodular operations, 
the previous example illustrates the variations of possible representations. 
In examples, which we gather in the Subsection \ref{Examplez}, the fraction-free
strategy leads to a moderate increase of coefficients.

On the other hand, switching to the polynomial framework changes the setup.
The algebra $R_*$ is not a principal ideal domain anymore, which
was the essential property for the existence of a diagonal form over $R$.
In the sequel, we show how that this problem can be 
resolved by introducing a suitable sorting condition
for the chosen module ordering.
Referring to the argumentation of Lemma \ref{triangular} yields 
the block-diagonal form \ref{BlockTriangle} with the 0 block above. 

\begin{footnotesize}
\begin{align}\label{BlockTriangle}
 \mathcal{G}({}_{R_*} M_*)=\left[ 
 \begin{array}{cccc}
 0 &    \dots &    \dots & 0 \\
  \fbox{*} &  &  &  \\
  \vdots &  & 0 &  \\
  * &  & \;\; &  \\
   & \fbox{*} &  \;\; &  \\ 
   & \vdots &  \;\; & 0 \;\;\;\\ 
   & * & \;\; \\
   &  & \ddots & \\
    &      &    & \fbox{*} \\
   &    *  &    & \vdots \\
   &      &    & * \\
 \end{array}
 \right].
\end{align}
\end{footnotesize}
Moreover, the rows with the boxed element have the smallest leading monomial with respect to the chosen ordering in the corresponding block. A block denotes all elements of the same leading position in $\mathcal{G}({}_{R_*}M_*)$.
In Theorem \ref{polynomialBasis} we show that 
these elements indeed generate ${}_R M$, while in Lemma \ref{MinDeg} we show that these elements provide us with additional information.
However, this result requires some preparations.
\end{remark}

\begin{lemma}\label{smallestElementInPoly}
Let $P$ be $R$ or $R_*$. For $M \in P^{p \times q}$ of full rank and for every $1 \leq i \leq q$, define $\alpha_i:=\min\{ \deg(a) \mid a \in {}_P M \setminus \{0\} \mbox{ and } \lpos(a)=i\}$. 
Then for all $1 \leq i \leq q$, there exists $h_i \in \mathcal{G}({}_P M)$ of degree $\alpha_i$ with $\lpos(h_i)=i$.
\end{lemma}
\begin{proof}
Recall that with respect to $<_*$, $\d$ is bigger than any monomial not involving $\d$.
Let $f \in {}_PM$ with $\lpos(f)=i$ and $\deg(f)=\alpha_i$. 
Suppose that for all $ g \in \mathcal{G}({}_P M)$ with 
leading position $i$, $\deg(g) > \alpha_i$ holds. Since $\mathcal{G}({}_P M)$ 
is a Gr{\"o}bner basis, there exists $g \in \mathcal{G}({}_P M) $  
such that $\lm(g)$ divides $\lm(f)$. This happens if and only if $\deg(g)\leq \deg(f)$ (because $R_*$ is a $G$-algebra and $R$ is an OLGAED), which yields a contradiction.
\end{proof}

The full rank assumption in the 
lemma guarantees the existence of $\alpha_i$ 
for each component $1 \leq i \leq q$. Note, that over $P=R_*$ the cardinality of
 $\{ \deg(a) \mid a \in {}_P M \backslash \{0\}$  and $\lpos(a)=i\}$ is
greater than one in general, hence there might be different selection strategies.
We propose to select an element according to $\min_{<_*}$
, see Lemma \ref{MinDeg}.
Recall the Lemma 3.3. from \citep{LS10}:

\begin{lemma}\label{triangular}
If one orders a reduced Gr\"obner basis in such a way, that
$\lm(\mathcal{G}({}_R M)_1) < \dots < \lm(\mathcal{G}({}_R M)_m)$,
then 
$\left[ \mathcal{G}({}_R M)_1, \ldots, \mathcal{G}({}_R M)_m \right]^T$
is a lower triangular matrix.
\end{lemma}

\begin{corollary}\label{smallestElementInPolyMat}
Lemma \ref{smallestElementInPoly} and Lemma \ref{triangular} yield
$$ \deg(\mathcal{G}({}_{R}M)_i)= \min\{ \deg(a) \mid a \in {}_R M\setminus \{0\} \mbox{ and } \lpos(a)=i\}.$$
\end{corollary}

\begin{lemma}\label{MinDeg}
Let $\alpha_i$ be the degree of the boxed entry with leading position
in the $i$-th column, that is
$$\alpha_i:=\deg( \, \min_{<_*}\{ \lm(b) \mid  b \in \mathcal{G}({}_{R_*}M_*) \mbox{ and } \lpos(b) =i \} 
\, ) .$$
Then for all $h \in {}_R M$ with $\lpos(h)=i$ we have $\deg(\lm(h)) \geq \alpha_i$.
\end{lemma}
\begin{proof}
Suppose that the claim does not hold and there is $h \in {}_R M$
with $\lpos(h)=i$ of degree smaller than $\alpha_i$. By Lemma
\ref{inMod}, there exists $a \in A_*$ such that 
$a h \in {}_{R_*}M_*$. Then $\deg(ah) = \deg(h)$ and 
$\lpos(ah)=i$. Due to Lemma \ref{smallestElementInPoly}, $\deg(f)\geq \alpha_i$
for all $f \in {}_{R_*}M_*$ with leading position $i$, hence we obtain a contradiction.
\end{proof}

\begin{corollary}\label{DegreeInv}
Lemma \ref{MinDeg} and Corollary \ref{smallestElementInPolyMat} imply, that 
for all $1 \leq i \leq q$
$$\min\{ \deg(a) \mid a \in {}_R M\setminus \{0\} \wedge \lpos(a)=i\}=
\min\{ \deg(a) \mid a \in {}_{R_*} M_*\setminus \{0\}  \wedge \lpos(a)=i\}.$$
\end{corollary}

\begin{theorem}\label{polynomialBasis}
Let $M \in R^{p \times p}$ be of full rank.  
For each $1 \leq i \leq p$, let $\alpha_i$ be as in Lemma \ref{MinDeg}.
Let us define $b_i$ to be the element from $\{ b \in \mathcal{G}({}_{R_*}M_*): \lpos(b) =i, \deg(b) = \alpha_i \}$ with the smallest leading monomial.
Then ${}_R\langle b_1, \dots, b_p \rangle = {}_R M$. 
Moreover, the set $\{b_1, \ldots, b_p\}$ corresponds to the subset of all rows 
with a boxed entry in the block triangular form \ref{BlockTriangle}.  
\end{theorem}
\begin{proof}
Since $M$ is of full rank, the minimum in the definition of $b_i$ exists for each $1 \leq i \leq p$.
Let $f \in {}_RM \backslash \{0\}$. Due to Corollary \ref{DegreeInv}, there exists $1\leq k \leq p$ such that
$\lpos(b_k)=\lpos(f)$ and $\deg(b_k)\leq\deg(f)$. Thus there exists an element
$s_k \in R$ such that $\deg( f- s_kb_k)<\deg(b_k)$. Since $f- s_k b_k \in {}_RM$,
Corollary \ref{DegreeInv} implies that we have $\lpos(f- s_k b_k)< \lpos(f)$. Iterating this
reduction leads to the remainder zero and thus $f=\sum_{i=1}^k s_ib_i$.
\end{proof}

\noindent
\textbf{Notation}. Using the notation of the previous theorem, let
$\G^*({}_R M):=\left[ b_1, \ldots, b_g \right]^T$, which is by construction a lower triangular matrix.
In the sequel, let $M\in R^{p \times p}$ be of full rank. 
Then $\G^*({}_R M)$ is a square matrix.

Recall that an involutive anti-automorphism (or an \textbf{involution}) $\theta$ of a ring $A$ is a $K$-linear map, satisfying $\theta(a b) = \theta(b) \theta(a)$ for all $a,b \in A$ and $\theta^2 = \id_A$.
Moreover, we define by $\widetilde{\theta}(M)$ the
application of an involution $\theta$ to the entries of the transpose of $M$.

\begin{proposition}\label{trafoMatrix}
Suppose $M \in R^{p \times p}$ is a full rank matrix and there is $U_* \in R_*^{\ell \times p}$
such that $U_* M_* = \mathcal{G}({}_{R_*}M_*)$. 
Let us select the indices
\begin{align}\label{indexChoice}
\{t_1, \dots, t_p\} \subseteq \{1, \dots, \ell\} \mbox{ such that } 
\{(U_* M_*)_{t_1}, \dots, (U_* M_*)_{t_p}\} 
= \G^*({}_R M)
\end{align}
Then $U:=[(U_*)_{t_1}, \dots, (U_*)_{t_p}]^T$ is $R$-unimodular in $R^{p \times p}$
and $U M_* = \G^*({}_R M)$.
\end{proposition}
\begin{proof}
The equality $U M_* = \G^*({}_R M)$ follows by the definition of $U$. Now we show that $U$ is $R$-unimodular. 
Note that $UM_* \subset R^{p \times p} \supset M_*$ and ${}_R(U M_*) = {}_R \G^*({}_R M) = {}_R M = {}_R M_*$ holds. Thus there exists $V \in R^{p \times p}$ such that
$M_*=V(UM_*)$. Then $VU = \id_{p \times p}$ and analogously $UV=\id_{p \times p}$ since $M$ has full row rank.
\end{proof}

\begin{lemma}\label{columnGeneratorPoly}
The equality of the following left ideals holds:
$${}_R\langle \theta(\G^*( {}_{R}M)_{p1}), \dots, \theta(\G^*( {}_{R}M)_{pp})\rangle=
{}_R \langle \G^* (\widetilde{\theta}( \G^*( {}_{R}M)))_{pp} \rangle. $$
\end{lemma}
\begin{proof}
Using the argumentation given in the proof of Lemma 3.4. of \citep{LS10},
we obtain
$${}_R\langle \theta( \G^*( {}_{R}M)_{p1}), \dots, \theta( \G^*( {}_{R}M)_{pp})\rangle=
{}_R \langle \mathcal{G}(\widetilde{\theta}( \G^*( {}_{R}M)))_{pp} \rangle.$$

\noindent
Because of ${}_{R} \G^*( {}_{R}M) = {}_{R}\mathcal{G}( {}_{R}M)$ we have
$\widetilde{\theta}( \G^*( {}_{R}M))_{R} = \widetilde{\theta}(\mathcal{G}( {}_{R}M))_{R}$ and thus \\
\mbox{${}_{R} \G^*(\widetilde{\theta}(\G^*( {}_{R}M))) = {}_{R}\mathcal{G}(\widetilde{\theta}(\mathcal{G}( {}_{R}M)))$}.
Since both $\mathcal{G}(\widetilde{\theta}(\mathcal{G}( {}_{R}M))$ and 
$G^*(\widetilde{\theta}( \G^*( {}_{R}M)))$ are lower triangular matrices,
with the latter identity above we obtain
${}_{R}\langle \mathcal{G}(\widetilde{\theta}(\mathcal{G}( {}_{R}M)))_{pp} \rangle=
{}_{R}\langle \G^*(\widetilde{\theta}(G^*( {}_{R}M)))_{pp} \rangle$.
\end{proof}

Now we are ready to formulate the fraction-free version of the Algorithm 
3.5 \textit{Diagonalization with Gr\"obner bases} from \citep{LS10}.

\begin{algorithm}[\texttt{Fraction-free diagonalization with Gr\"obner Bases}]  
\label{diagonalPoly}
\begin{algorithmic}
\STATE 
\REQUIRE $M \in R^{p \times p}$ of full rank, $\theta$ an involution on $R_*$ and $\widetilde{\theta}$ as above.
\ENSURE $R$-unimodular matrices $U, V, D \in R_*^{p \times p}$ such that $U \cdot M \cdot V = D = \Diag ( r_1, \dots, r_p )$.
\STATE  Find $T \in R^{p \times p }$ unimodular such that $T M \in R_*^{p \times p}$
\STATE   $M^{(0)} \leftarrow T M$, \quad $U \leftarrow T $, \quad $V \leftarrow \id_{p \times p} $ 
\STATE  $i \leftarrow 0$ 
 \WHILE{$M^{(i)}$ is not a diagonal matrix { \bf or } $i\equiv_2 1$}
\STATE     $i \leftarrow i+1$
\STATE    Compute $U^{(i)}$ so that  $U^{(i)} \cdot M^{(i-1)} = \mathcal{G}({}_{R_*} M^{(i-1)}) \in R_*^{\ell \times p}$
\STATE     Select $\{t_1, \dots, t_p\} \subseteq \{1, \dots, \ell\}$ as in (\ref{indexChoice}) of Prop. \ref{trafoMatrix} 
\STATE     $ U^{(i)} \leftarrow [(U^{(i)})_{t_1}, \dots, (U^{(i)})_{t_p} ]^T$ 
\STATE     $M^{(i)} \leftarrow \widetilde{\theta}( G^*({}_R M) )$
              \IF{$i \equiv_2 0$}
                   \STATE  $V \leftarrow V \cdot \widetilde{\theta}(U^{(i)}) $
                    \ELSE \STATE $ U \leftarrow U^{(i)} \cdot U $
              \ENDIF
   \ENDWHILE
\RETURN $(U,V,M^{(i)})$
\end{algorithmic}
\end{algorithm}

\begin{remark}
\label{unimod}
It is important to mention, that the matrices $U,V,D$ (hence the elements $r_i$ as well) have entries from $R_*$, that is, they are polynomials. However, $U$ and $V$ are only unimodular over $R$ and, in general, they need not be unimodular over $R_*$ for obvious reasons. In  Subsection \ref{Examplez} we will investigate, over which subalgebras of $R$ the matrices $U$ or $V$ become unimodular.
\end{remark}

\begin{theorem}
 Algorithm \ref{diagonalPoly} terminates with the correct result.
\end{theorem}
\begin{proof}
The proof is a natural generalization of the proof of the
Theorem 3.6. from \citep{LS10}. At first, we 
use Proposition \ref{trafoMatrix}. Moreover, 
Lemma \ref{columnGeneratorPoly} provides a replacement for
the arguments we used in the Lemma 3.4. of \citep{LS10}. 
\end{proof}

Algorithm \ref{diagonalPoly}, as well as the original algorithm of \citep{LS10}, can be extended to $M \in R^{p \times q}$ along the lines already 
discussed in Remark 3.7
of \citep{LS10}. Our implementation (cf. Section \ref{Impl}) works for arbitrary matrices. 

As for examples, a $2\times 2$ matrix over the Weyl algebra has been considered in detail in Example 3.8 of \citep{LS10}. Note that a fraction-free method was used indeed.

\begin{example}\label{runningEx3}
Consider the first shift algebra $R_* = S_1 = \K\langle x,s \mid sx=xs+s \rangle$
and its localization (often called the first rational shift algebra) 
$R =\K(x)\langle s \mid sx=xs+s \rangle$. There are precisely two involutions, which can be presented by diagonal matrices on $R_*$, namely $x \mapsto -x, s\mapsto -s$ and $x \mapsto -x, s\mapsto s$. Let us take the latter and call it $\theta$.
Consider the matrix in $R_*^{2\times 3}$
\[
M = 
\left[
\begin{array}{*{5}{c}}
(x-1)s+x^{2}-x &  & xs+x^{2} &  & (x+2)s+x^{2}+2x \\
s+x & & 0 &  & s
\end{array}
\right].
\]

\noindent
As we can see, $T=\id_{2 \times 2}$ and thus $M^{(0)}:=M,$ $U=\id_{2 \times 2}$, $V=\id_{3 \times 3}$ and $i=0$. \\
\noindent
Since $M$ is not a square matrix, under the \textit{while} condition ``while $M^{(i)}$ is not a diagonal matrix'' in the Algorithm we mean the following. 
The computation will run until the matrix we obtain contains a diagonal square submatrix and the entries outside of this submatrix are zero.

{\bf 1:} Since $M^{(0)}$ is not diagonal, we enter the while loop. $i:=1$.

\[
M' := \mathcal{G}({}_{R_*}M^{(0)}) =
\begin{footnotesize}
\left[
\begin{array}{*{3}{c}}
-3s^{2}-(x^{2}+7x+6)s-x^{3}-4x^{2}-3x & (x+1)s^{2}+(x^{2}+2x+1)s & 0 \\
-3s-3x & xs+x^{2} & x^{2}+2x
\end{array}
\right]
\end{footnotesize}
\]

We set $U:=U_1$, where $U_1 M^{(0)} = M'$ and 
\begin{footnotesize}
\[
U_1 = 
\left[
\begin{array}{*{3}{c}}
s & -(x+3)s-x2-4x-3 \\
1& -x-2
\end{array}
\right].
\]
\end{footnotesize}

Moreover, $M^{(1)}:=\widetilde{\theta}(M')\in R_*^{3\times 2}$.


{\bf 2:} Since $M^{(1)}$ is not diagonal, we enter the while loop. $i:=2$.


\[
M^{(1)} =
\begin{footnotesize}
\left[
\begin{array}{*{3}{c}}
-3s^{2}-x^{2}s+5xs+x^{3}-4x^{2}+3x & -3s+3x \\
-xs^{2}-s^{2}+x^{2}s & -xs-s+x^{2} \\
0 & x^{2}-2x
\end{array}
\right]
\end{footnotesize}
\]

\[
M' := \mathcal{G}({}_{R_*}M^{(1)}) =
\begin{footnotesize}
\left[
\begin{array}{*{2}{c}}
0 & 0 \\
4x^{4}+12x^{3}-4x^{2}-12x & 0 \\
 -4xs-4s & -4x 
\end{array}
\right]
\end{footnotesize}
\]

The transformation matrix $U_2\in R_*^{3\times 3}$ is dense,
so we show its highest terms with respect to $s$:

\begin{footnotesize}
\[
U_2 = 
\left[
\begin{array}{*{3}{c}}
-(x^{2}+5x-6)s^{2} + \ldots \ & -7(x-1)s^{2}+\ldots \ & (-x+1)s^{2}+\ldots \\
-3(x+6)s^{2}+\ldots \ & -21s^{2}+\ldots \  & -3s^{2}+\ldots \ \\
(x^{2}+5x-6)s^{2}+\ldots \ & 7(x-1)s^{2}+\ldots \ & (x-1)s^{2}+\ldots 
\end{array}
\right]
\]
\end{footnotesize}

Moreover, we put $M^{(2)}:=\widetilde{\theta}(M')\in R_*^{2\times 3}$ and 
$V:=\widetilde{\theta}(U_2)$.

{\bf 3:} Since $M^{(2)}$ is not diagonal, we enter the while loop. $i:=3$.

\[
M^{(2)} = 
\left[
\begin{array}{*{3}{c}}
0 & 4x^{4}+12x^{3}-4x^{2}-12x & -4xs-4s \\
0 & 0 & -4x
\end{array}
\right]
\]

\[
M' := \mathcal{G}({}_{R_*}M^{(2)}) =
\begin{footnotesize}
\left[
\begin{array}{*{3}{c}}
0 & 4(x^{4}+3x^{3}-x^{2}-3x) & 0 \\
0 & 0 & 4x
\end{array}
\right]
= U_3 M^{(2)}, \text{ where  }
U_3 = 
\left[
\begin{array}{*{2}{c}}
1 & -s \\
0 & -1
\end{array}
\right].
\end{footnotesize}
\]

Thus we define $M^{(3)}:=\widetilde{\theta}(M')\in R_*^{2\times 3}$ and 
$U:=U_3 \cdot U$.

\textbf{4:}. Since $M^{(3)}$ is diagonal (that is, consists of a diagonal 
submatrix and the rest of entries are zeros) but $i = 1 \mod 2$, we do one
more run, which finishes and returns the final data:

\[
D = \left[
\begin{array}{*{3}{c}}
0 & (x-1) x (x+1) (x+3)  & 0 \\
0 & 0 & x
\end{array}
\right], \;
U = \frac{1}{4}
\left[
\begin{array}{*{3}{c}}
0 & -(x+1)(x+3) \\
1 & -(x+2)
\end{array}
\right],
V = 
\]

\[
\left[
\begin{footnotesize}
\begin{array}{lll}
-(x^{2}+5x-6)s^{2}-(x^{3}+4x^{2}-x-4)s & -7(x-1)s^{2}-(3x^{2}+5x-8)s-4x+4 & v_{13}
\\
-3(x+6)s^{2}-(x^{3}+7x^{2}+5x-13)s-x^{4}-2x^{3}+7x^{2}+8x-12 & 
-21s^{2}-(7x^{2}+2x-11)s-3x^{3}-2x^{2}+13x-8 & v_{23} \\
(x^{2}+5x-6)s^{2}+(2x^{3}+6x^{2}-14x+6)s+x^{4}-7x^{2}+6x & 7(x-1)s^{2}+(10x^{2}-16x+6)s+3x^{3}-4x^{2}-5x+6 & v_{33}
\end{array}
\end{footnotesize}
\right],
\]
where \begin{footnotesize}
$v_{13} =(x-1)s^{2}+(x^{2}-x)s$, $v_{23} =  3s^{2}+(x^{2}+2x-5)s+x^{3}-2x^{2}-3x+8$\end{footnotesize} 
and
\begin{footnotesize} 
$v_{33} = (-x+1)s^{2}-2(x-1)^{2}s-x^{3}+4x^{2}-5x+2$
\end{footnotesize}.
Indeed, since both nonzero diagonal entries are units in $K(x)$, the output matrix can be further reduced to
\[
D' = \left[
\begin{array}{*{3}{c}}
0 & 1 & 0 \\
0 & 0 & 1
\end{array}
\right].
\]
However, then one has to divide explicitly by polynomials in $x$ in transformation matrices. We are not going to do this. Moreover, in what follows we will show, how to get important information from such matrices containing units from the non-constant ground field. As for the concrete example, we conclude, that $R^3/R^2 M \cong R$, thus a system module of $M$ is free of rank 1 over $R$.
\end{example}

\section{Solving Systems of Operator Equations}
\label{Solving}

\begin{remark}
\label{Malgrange}
Let us settle the terminology. Let $A$ be the $\K$-algebra of $\K$-linear \textit{operators}.
Consider a system of equations in unknown functions $\omega_1,\ldots,\omega_m$. A system is linear, if
it can be written in the matrix notation, that is
$S \cdot [\omega_1,\ldots,\omega_m]^T = 0$, where
$S$ is a rectangular matrix with entries from $A$. Then one
associates to $S$ a left $A$-module $\cal{M}$, which is finitely presented 
by the matrix $S$. Given a left $A$-module $\cal{F}$, we usually speak
of \textit{solutions of} $\cal{M}$ in $\cal{F}$. 
The celebrated Lemma of B.~Malgrange tells us, that the solutions to a linear system of equations $S$ in a left $A$-module $\cal{F}$ are in one-to-one correspondence with the elements of the abelian group $\Hom_A({\cal M}, {\cal F})$.
This allows us to avoid the reference to a specific solution space by
adressing an abstract one. However, in the context of this paper
we address 
the space of distributions (though yet more general hyperfunctions fit into our framework as well) as the space of more general solutions in addition to meromorphic functions.

Assume there is a bigger function space ${\cal G}$, which is an $A$-module. 
Thus we have an exact sequence $0 \to {\cal F} \to {\cal G}$. Due to the left exactness of the $\Hom$ functor, we obtain that $0 \to \Hom_A({\cal M}, {\cal F}) \to \Hom_A({\cal M}, {\cal G})$ is an exact sequence as well. This justifies the fact, that working with a more general solution space ${\cal G}$ we obtain not less solutions to a module ${\cal M}$ as with ${\cal F}$.
\end{remark}

An $R$-module is naturally an $R_*$-module, but an $R_*$-module is not necessarily an $R$-module. There can be $R_*$-modules $M$ with $S$-torsion, that is those for which $S^{-1} M = 0$ holds. 
Hence, though working over $R$ brings significant comfort, we are interested
in gaining more information from the algebraic structure by looking at 
$R_*$ and, more generally, subalgebras between $R_*$ and $R$. Since there
are $R_*$-modules, which are not $R$-modules, in view of Malgrange's Lemma
there might be more solutions in $R_*$-modules as in $R$-modules. See 
Examples \ref{EulerCont} and \ref{EulerDiscrete} for illustrative details.

We need to recall and develop mechanisms, which allow us to tackle 
localizations of operator algebras. We will show, how different small
enough localizations look for common operator algebras.  

\subsection{Ore Multiplicative Closure}
\label{MCOre}
Let $W \in R^{r\times r}$ be a square matrix with entries in $R$. Assume that
there exists left inverse matrix $T$, that is $TW=\id_{r\times r}$. Then by
Lemma \ref{inMod} there exists 
a diagonal matrix $Q=\Diag(\ldots,q_{ii},\ldots)$ such that $Q$ (resp. $QT$) has entries from $A_*$ (resp. $R_*$). Let $\Omega = \{ q_{ii} \mid 1\leq i \leq r$, $q_{ii} \not\in\K\} \cup \{1\}$. Denote by $R_W$ the Ore localization of $R_*$ with respect to a multiplicatively closed Ore set $S_W$, which is defined as follows. 
Let $\mathcal{M}(\Omega)$ be the two-sided multiplicative closure of $\Omega$, equivalently the (possibly non-commutative) monoid, generated by a
finite set $\Omega\subset R_*\setminus\{0\}$. Let $S_W$ be an \textit{Ore closure} of $\mathcal{M}(\Omega)$ in $R_*$, that is a set, containing $\mathcal{M}(\Omega)$, which is an Ore set in $R_*$. Such a closure always exists since $R$ is 
a localization of $R_*$ with respect to $S=A_*\setminus\{0\}$, but we are interested in computing a closure, which is minimal in the sense that there
exists no closure $S'$ satisfying $\mathcal{M}(\Omega) \subsetneq S' \subsetneq S_W$.

\subsection{Ore Closure in Classical Algebras}


In order to localize a 
domain $R$ with respect to a multiplicatively closed set $S$, the latter needs to be a (left and right) Ore set in $R$. Thus, we
are going to show, how to compute an Ore closure of $S$ in $R$. Moreover,
we ask for small generating sets of a monoidal localization.

Let us recall the well-known Lemma (see e.~g. \cite{ZS}) first.
\begin{lemma}
Let $f\in \K[x]:=\K[x_1,\ldots,x_n]$ be a non-constant monic polynomial and $S =\{f^i \mid i\in\N_0 \}$. 
Moreover, let $f = f_1^{d_1} \cdots f_r^{d_r}$ be an irreducible factorization in $\K[x]$
and $T := \{ f_1^{\alpha_1}, \ldots, f_r^{\alpha_r} \mid \alpha \in\N_0^r \}$.
Then $S^{-1}\K[x] = T^{-1}\K[x]$.
\end{lemma}

\noindent
Thus, if $f=(x+1)^2y^3$, then $S=\{ (x+1)^{2i} y^{3i} \mid i\in\N_0\}$ and $T=\{ (x+1)^i,y^i \mid i\in\N_0\}$. \\

Let $S$ be a multiplicatively closed subset of a domain $R$ and $0\not\in S$. In order to prove, that $S$ is an Ore set in $R$, one has to show, that $\forall \ (s,r) \in S \times R$, there exist $(t,q) \in S \times R$,
such that $r \cdot t =  s \cdot q$ in $R$ 
and $\forall (t,q) \in S \times R$ there exist $(s,r) \in S \times R$, satisfying the same condition. In other words, one can rewrite any left (resp. right) fraction as a right (resp. left) fraction.

We will analyze the smallest nontrivial Ore sets in the first Weyl, shift and $q$-commutative algebras (these results seem to be folklore) and give simple constructive proofs of the Ore property for them.

\begin{lemma}\label{OreWeyl}
 Let $A_1$ be the first Weyl algebra and $f \in\K[x]\setminus\K$.
Then $S =\{f^i \mid i\in\N_0 \}$ is an Ore set in $A_1$.
\end{lemma}
\begin{proof}
We have for $f^{i+1}$ and $i\in\N$, that $\d \cdot f^{i+1} = f^{i} \cdot (f \d + (i+1) \tfrac{\d f}{\d x})$.
By induction, one can prove, that for $j\in\N$ and $i+1 \geq j$ one has
$\d^j \cdot f^{i+1} = f^{i-j+1} \cdot (f^j \d^j + v_{ij})$, 
where the terms of $v_{ij}\in A_1$ have degree at most $j-1$ and contain derivatives up to $f^{(j)}$.
Suppose we are given $g = \sum_{j=0}^d b_j(x) \d^j \in A_1$ with $b_d\not=0$ and $f^k$ for a fixed $k\in\N$. Then
\[
\d^j \cdot f^{d+k} = \d^j \cdot f^{j+k} \cdot f^{d-j} = 
f^{k}(f^j \d^j + v_{j+k,j}) f^{d-j}
\]
Thus
\[
g \cdot f^{d+k} = \sum_{j=0}^d b_j(x) \d^j \cdot f^{d+k} = 
f^{k}
\cdot \sum_{j=0}^d b_j(x) (f^j \d^j + v_{j+k,j}) f^{d-j}.
\]
\end{proof}

Consider now the first shift algebra $S_1$ (cf. Example \ref{runningEx3}). Since $sx = (x+1)s$, we see that for all $z\in\Z$
\[
(x+z+1)^{-1} s = s (x+z)^{-1} 
\]
thus it seems natural to have polynomials with integer shifts of their argument in the set $S$ as above in addition to $S$ itself.

\begin{remark}
For $f\in K[x]\setminus\K$, the set $S=\{f^i \mid i\in\N\}$ is not an Ore set in the first shift algebra.
Take $s$ and $f^k(x)\in S$, we're looking for $f^\ell(x)$ and $t\in S_1$, such that $s f^{\ell}(x) = f^k(x) t$. The left hand side is $f(x+1)^{\ell} s$ thus $f^k(x) t = f(x+1)^{\ell} s$. But $f(x) \nmid f(x+1)$ for non-constant $f$, thus there exists no $t\in S_1$ satisfying the latter identity. It means, that we have to enlarge $S$ in order to obtain an Ore set in $S_1$.
\end{remark}

\begin{lemma} \label{OreShift}
Let $S_1$ be the first shift algebra (cf. Example \ref{runningEx3}) and $f \in\K[x]\setminus\K$.
Then $S=\{ f^n(x\pm z) \mid n,z\in\N_0\}$ is an Ore set in $S_1$.
\end{lemma}
\begin{proof}

Given $g = \sum_{j=0}^d b_j(x) s^j \in S_1$ with $b_d\not=0$ and $h(x) = f^k(x+z_0) \in S$ with $k\in\N, z_0\in\Z$, let us define 
$g_f(x) := \prod_{i=0}^d h(x-i)\in S$. Then

\[
g \cdot g_f(x) = 
\sum_{j=0}^d b_j(x) s^j \cdot \prod_{i=0}^d h(x-i) = 
h(x)^d  
\cdot \sum_{j=0}^d b_j(x) \bigl(\prod_{i=0, i\not=j}^d h(x+j-i) s^j \bigr).
\]
\end{proof}

A similar phenomenon can be observed in quantum algebras as well.

\begin{lemma} Let $Q_1$ be the first $q$-commutative algebra $\K(q) \langle x,s \mid y x = qx y \rangle$ and $f \in\K[x]\setminus\K$.
Then $S=\{ f^n(q^{\pm z}x) \mid n,z\in\N_0\}$ is an Ore set in $Q_1$.
\end{lemma}
\begin{proof}
Note, that for any $g(x)\in\K[x]$ one has $y^m g(x) = g(q^m x) y$. 
Suppose we are given $g = \sum_{j=0}^d b_j(x) y^j \in Q_1$ and 
$h(x) = f^k(q^{\ell}x) \in S_1$. Let us define 
$g_f(x) := \prod_{i=0}^d h(q^{-i}x)\in S$. Then
\[
g \cdot g_f(x) = 
\sum_{j=0}^d b_j(x) y^j \cdot \prod_{i=0}^d h(q^{-i}x) = 
h(x) 
\cdot \sum_{j=0}^d b_j(x) \prod_{i=0, i\not=j}^d h(q^{j-i}x) y^j.
\]
\end{proof}

Let $R'$ be a $\K$-algebra with $R_* \subseteq R' \subseteq R$, where $R_*,R$ are as above and $M\in R^{m\times m}$. 
Consider a system of equations $M \omega = 0$ in unknown functions $\omega = (\omega_1,\ldots,\omega_m)$ from a space of functions $\kf$, which will possess some module structure, see below.
Assume, that we have computed $U,V$ (unimodular over $R$) and $D=\Diag(d_{11},\ldots,d_{mm})$ satisfying $UMV=D$.

\subsection{$\kf$ is an $R$-module}

Inverting $V$, we obtain $UM = DV^{-1}$, hence $M\omega=0$ is equivalent to $0=UM\omega=DV^{-1}\omega$. Thus, introducing an $R$-automorphism of $\kf$, defined by $\varpi := V^{-1} \omega$, we obtain a decoupled system $\{ d_{ii} \varpi_i = 0 \}$. Note that if $d_{ii}=0$, then $\varpi_i$ is called a free variable of the system, e.~g. in \citep{Eva05}. The solutions of the decoupled system in $\kf$ are precisely the solutions of $M \omega = 0$ in $\kf$.

\subsection{$\kf$ is an $R_*$-module}

Indeed, the $R$-automorphism of $\kf$ above can be defined as soon as $V$ is invertible. We can regard this as a kind of ``analytic'' transformation of $\kf$. 

\begin{proposition}
Let $UMV=D$ as before. Let $S_U$ and $S_V$ be Ore multiplicative closures of $U$ and $V$ respectively, according to Sect. \ref{MCOre}. Moreover, let $S$ be an Ore multiplicative closure of the monoid $S_V \cup S_U$. 
Then, over $S^{-1} R_* \subset R$ we have $M \omega=0 \Leftrightarrow D (V^{-1} \omega) = 0$. Thus it is possible to decouple the system $M$. Moreover, there might be solutions in an $S^{-1} R_*$-module ${\cal G}$.
\end{proposition}
\begin{proof}
Let $R_V = (S_V)^{-1}R_*$. Then on an $R_V$-module $\kg$ we can define an $R_V$-automorphism $\omega \mapsto V^{-1} \omega$. Thus over $R_V$ we have $UM \omega=0 \Leftrightarrow D \varpi = 0$.
Moreover, $U$ is invertible over $(S_U)^{-1}R_*$.
Hence, over $S^{-1} R_*$ we have $M \omega=0 \Leftrightarrow UM \omega=0 \Leftrightarrow D \varpi = 0$.
\end{proof}

\begin{corollary}
With notations of the Proposition, there is an explicit isomorphism of $S^{-1} R$-modules
\[
(S^{-1}R)^{m} / (S^{-1}R)^n M  \cong (S^{-1}R)^{m} / (S^{-1}R)^n UMV = (S^{-1}R)^{m} /(S^{-1}R)^n  D .
\]
\end{corollary}

Note that the left transformation matrix $U$ also can contain essential information about the so-called singularities of a system, which often are connected to meromorphic (and non-holomorphic) solutions.

\begin{example}
\label{EulerCont}
Over the first Weyl algebra $A_1 = R_*$, consider the single equation $(x\d) \omega = 0$. Here $U=x$ and thus $S_U = \{ x^i \mid i\in\N_0\}$. So $R_U = (S_U)^{-1} R_* \cong \K[x,x^{-1}]\langle \d \rangle \subsetneq \K(x)\langle \d \rangle$ and over $R_U$, the equation is equivalent to $\d \omega = 0$, whose solutions in any nonzero $R$-module $\kf$ contain $\K$.

On the other hand, the division by $x$ over $R_*$ is not allowed. Consider an $R_*$-module $\cal{D}'(\R)$ of distributions. Then,
we see that a 
$\K$-multiple of the Heaviside step function is a solution. Thus, $\K\cdot H(x) \oplus \K$ is a subspace of solutions of the system $(x\d) \omega = 0$.
\end{example}

\begin{example}
\label{EulerDiscrete}
Consider the univariate sequence space ${\cal S}$, that is the $\K$-vector space of all functions $f:\Z \to \K$, which is an $S_1(\K)$-module.
Recall, that a discrete analogue $H(n)$ of the Heaviside step function is defined to be 0, if $n<0$ and 1 otherwise. 

Consider the analogon of the equation above $(n(s-1))\omega = 0$ 
over the shift algebra $S_1(\K) = R_*$, cf. Example \ref{runningEx3},
where $s$ acts as $(s\omega)(n) = \omega(n+1)$.
The invertibility of $n$ is reflected in the localization with respect to $\{ (n \pm k)^m \mid k,m \in \N_0\}$ (cf. Lemma \ref{OreShift}) and implies that 
there is 1-dimensional subspace of constant solutions.

Over $R_*$, $n$ is not invertible. Moreover, the ideal $\langle n \rangle \subset R_*$ annihilates any constant multiple of the Kronecker delta $\delta_{n,0} \in {\cal S}$. 
Since $\delta_{n,0} = H(n) - H(n-1) = (s-1)(H(n-1))$, another set of solutions to the considered equation are 
$\K$-multiples of $H(n-1)$.  
Hence, as one can easily check, $\K\cdot H(n-1) \oplus \K \subset {\cal S}$ is indeed the whole set of solutions to the equation $(n(s-1))\omega=0$.
\end{example}

\section{Cyclic Vector Method}
The existence of the Jacobson form of a matrix over a simple Euclidean domain \citep{Cohn, Jacobson} is a very strong result. In particular, for a square matrix $M$ of full rank over $R$, a Jacobson form is $\Diag(1,\ldots,1,r)$ for some $r\in R\setminus\{0\}$. Then a module, presented by $M$ is isomorphic to a
cyclic module and its presentation is a principal ideal.
The method of finding a cyclic vector in a finitely presented module and obtaining a left annihilating ideal for it is also used in $D$-module theory. 

\begin{proposition}
\label{cyclicAlg}
Let $R$ be a simple OLGAED, representable as $A[\d;\sigma,\delta]$ with a
division ring $A$ over the field $\K$.
Let $M = \Diag(m_1,\ldots,m_r)$ be a full rank $r\times r$ matrix, that is $m_i\not=0$. Then $d:=d(M) = \sum \deg(m_i)$ is an invariant of the module 
$R^{r}/R^{r}M$, since it is the dimension of the module over the division ring $A$.
Let $p = [p_1,\ldots,p_r]^T\in R^{r\times 1}$ and $c\in R$ be a generator of the left annihilator ideal of $p$ in the module $R^{r}/R^{r}M$. If $\deg(c) = \sum \deg(m_i)$, then $\Diag(1, \dots, 1, c)$ is a Jacobson form of $M$.
\end{proposition}

\begin{proof}
There is an $R$-module homomorphism $\varphi_p$ and the corresponding induced exact sequence
\[
0 \longrightarrow R/\langle c \rangle = R/ \ker \varphi_p \overset{\cdot p}{\longrightarrow} R^{r}/R^{r}M \longrightarrow \coker \varphi_p \longrightarrow 0.
\]
Since all $R$-modules above are finite dimensional over $A$, the dimension of $\coker \varphi_p$ is precisely $d - \deg(c)$. Hence if $\deg(c)=d$, then
\[
R^{r}/R^{r} \Diag(1,\ldots,1,c) \cong R/\langle c \rangle \cong R^{r}/R^{r}M = R^{r}/R^{r} \Diag(m_1,\ldots,m_r).
\]
\end{proof}

\begin{remark}
By using the previous Proposition, we propose the following probabilistic approach for the computation of a cyclic presentation of a module. We use the
dimension $d$ of $M$ as the certificate.
For every $1\leq i \leq r$, consider a polynomial $p_i$ of degree at most $\deg(m_i)-1$ in $\d$ with random coefficients from $A$. Compute the generator $c\in R$ of 
$\ker \varphi_p$.  
If $\deg c = d$, we are done. 
Otherwise (it means that the image of $\varphi_p$ is a proper submodule) one takes another set of random polynomials $p_i$ and repeats the procedure.
In order to turn this approach into an algorithm, one needs to obtain probabilistic estimations on the length of random coefficients like in \citep{Kaltofen} and \citep{SL97}. However, it is more complicated in the case when $A$ is non-commutative division ring and needs to be investigated in more detail. 
Our experiments, illustrated by the examples below, detect a considerable coefficient swell both in the intermediate computations and in the output. 
Thus we conclude, that the deterministic method by Leykin \citep{Leykindiss} 
is better for the case, when random vectors contain polynomials from $\K[x]$
as coefficients.  

Yet another application of this approach is the search for a proper submodule of a module. Namely, if $\deg c < d$, the image of $\varphi_p$ is such a submodule.
\end{remark}

In the following examples we work in the first rational Weyl algebra $A$. 
Consider two $3\times 3$ diagonal matrices with polynomials of degrees $1,2,3$ 
in $\d$ at the diagonal. We call these matrices $M_1$ and $M_2$.
By the remark above, in order to find a cyclic vector of the corresponging module, it is enough to consider a vector $p$ of the form $[c(x), a(x) \d + b(x), u(x) \d^2 + v(x) \d + w(x)]^T$ for $a,b,c,u,v,w\in\K[x]$. We generate the latter polynomials in a random way, taking their coefficients from $\Q$ from the range $[0,\ldots,100)$. The vector $p_1$ has coefficients of $x$-degree at most $3$ and the vector $p_2$ of $x$-degree at most $4$:

\noindent
$p_1 = $ \begin{footnotesize}
$[98x^{3}+4, (2x^{2}+17)\d+87x^{3}, (98x^{2}+11x) \d^{2}+(8x^{3}+62x^{2}+31)\d+89x]^T $
\end{footnotesize} , \\
$p_2 = $ \begin{footnotesize}
$[ 50x^{4}+13x^{3}+97x^{2}, (25x^{4}+91x+72)\d+53x^{4}+96x^{3}+90x, (41x^{3}+57x^{2})d^{2}+(36x^{4}+53x^{2}+54x+83)\d+2x^{3}+87]^T$.
\end{footnotesize}

By computing kernels as in Prop. \ref{cyclicAlg} we obtain generators $c^i_1,c^i_2$ for $i=1,2$ respectively.
Since $\deg c^i_j = 6 = \sum_{k=1}^3 \deg M_{kk}$, it follows that $p_j$ are cyclic vectors for both matrices $M_1,M_2$ and the Jacobson form of $M_i$ with respect to $p_j$ is $\Diag(1,1,c^i_j)$.

\begin{example}
Consider the matrix $M_1 = \Diag(\d,x \d^{2}+2 \d, x^{2} \d^{3}+4x \d^{2}+2 \d)$.
For the vector $p_1$ 
we obtain the cyclic generator $c^1_1=$
\begin{footnotesize}
$(1011752x^{8}-348435x^{7}-846320x^{5}-2965480x^{4})\d^{6}+(9105768x^{7}-3484350x^{6}-10155840x^{4}-38551240x^{3}) \d^{5}+(15176280x^{6}-6271830x^{5}-25389600x^{3}-115653720x^{2}) \d^{4}-35585760x \d^{3}+35585760 \d^{2}$.
\end{footnotesize}

For the vector $p_2$ 

we obtain the cyclic generator $c^1_2=$ 
\begin{footnotesize}
$(395647200x^{16}-264348000x^{15}+72499244880x^{14}-432343555560x^{13}+267049574380x^{12}+8987298499008x^{11}-45076322620512x^{10}+91959270441432x^{9}-24315286945590x^{8}+20084358713472x^{7}-19034034270714x^{6}-3963517931016x^{5}+240924562515x^{4}+239689051938x^{3}) \d^{6}+ \ldots$
\end{footnotesize}.

\end{example}

\begin{example} 
Consider the matrix $M_2 = \Diag(\d,x \d^{2}+2 \d, x^{2} \d^{3}+3x \d^{2}+\d)$.
For the vector $p_1$ 
we obtain the cyclic generator $c^2_1=$
\begin{footnotesize}
$(8352x^{12}+149292x^{11}+3213954x^{10}-8623701x^{9}+46759968x^{8}+18251056x^{7}+12525848x^{6}+1085484x^{5}+968320x^{4}) \d^{6}+\ldots$
\end{footnotesize}.

For the vector $p_2$ 
we obtain the cyclic generator $c^2_2 =$
\begin{footnotesize}
$(544946x^{21}-8586000x^{20}+4018843296x^{19}-2355826816422x^{18}+7580096636636x^{17}-91394273346228x^{16}+394488979119486x^{15}-1039358677414560x^{14}+2049350822715951x^{13}-6702303668155704x^{12}+15453420608607570x^{11}-9398963461913820x^{10}-1453404219438726x^{9}-83615730159492x^{8}+313903596685743x^{7}+646008054793470x^{6}+64626268386222x^{5}-13528366907400x^{4}-3553905268080x^{3}) \d^{6} + \ldots$
\end{footnotesize}.

\end{example}
\noindent
\textbf{Analysis of the data}. 
Since the computed cyclic generators $c^i_j$ are fraction-free, we present them
as polynomials from $\Z[x,\d]$.
We determine the following data and put them into the table.
\begin{itemize}
\item \texttt{NT} number of terms, \texttt{TD} total degree;
\item \texttt{BC}, \texttt{SC} and \texttt{AC} stand for 
the biggest resp. the smallest  resp. the arithmetic average of absolute values of coefficients; 
\item \texttt{BX}, \texttt{SX} and \texttt{AX} stand for 
the biggest resp. the smallest  resp. the arithmetic average of 
degrees in $x$ in every monomial $x^a \d^b$.
\end{itemize}

${}_{}$\\

\begin{center}
\begin{tabular}[h]{|l|c|c|c|c|c|c|c|c|}
\hline
Poly & \texttt{NT} & \texttt{TD} & \texttt{BC} & \texttt{SC} & \texttt{AC} & \texttt{BX} & \texttt{SX} & \texttt{AX}  \\
\hline
$c^1_1$ & 14 & 14 & 1.15$\cdot 10^8$ & 3.48$\cdot 10^5$ & 2.14$\cdot 10^7$ & 8 & 0 & 4.3 \\
\hline
$c^1_2$ & 85 & 22 & 3.2$\cdot 10^{15}$ & 2.6$\cdot 10^{8}$ & 2.3$\cdot 10^{14}$ & 16 & 0 & 6.8 \\
\hline
$c^2_1$ & 43 & 18 & 3.9$\cdot 10^8$ & 8.3$\cdot 10^3$ & 6$\cdot 10^7$ & 12 & 0 & 5.8 \\
\hline
$c^2_2$ & 126 & 27 & 4.95$\cdot 10^{17}$ & 5.4$\cdot 10^5$ & 2.35$\cdot 10^{16}$ & 21 & 0 & 9.4 \\
\hline
\end{tabular}
\end{center}

${}_{}$\\

Let us see, what can happen when the random polynomial coefficients from 
above are taken to be just numbers from $\K$.

\begin{example}
For the matrix $M_2$, let us take 5- and 10-digit random integers. The corresponding
vectors are then $p_1 =$ 
\begin{footnotesize} $[ 5535, 3892\d+20690, 6069\d^{2}+20660\d+17323]^T$ 
\end{footnotesize}
and $p_2 = $
\begin{footnotesize} $[1109725034, 618308146\d+1684065511, 2034108815\d^{2}+1702526110d+1361184996]^T$ 
\end{footnotesize}.
The corresponding cyclic generators are then
$c_1 =$ \begin{footnotesize} $(519739491811x^{10}+\ldots)\d^4+\ldots$ \end{footnotesize}
and $c_2 =$ \begin{footnotesize}$(689281531286620056404507706x10+\ldots)\d^4+\ldots$ \end{footnotesize}.

The coefficient data is put into the table with the same notations as before.

${}_{}$\\

\begin{center}
\begin{tabular}[h]{|l|c|c|c|c|c|c|c|c|}
\hline
Poly & \texttt{NT} & \texttt{TD} & \texttt{BC} & \texttt{SC} & \texttt{AC} & \texttt{BX} & \texttt{SX} & \texttt{AX}  \\
\hline
$c_1$ & 20 & 14 & 2.14$\cdot 10^{13}$ & 4.77$\cdot 10^{10}$ & 5.48$\cdot 10^{12}$ & 10 & 3 & 6.5 \\
\hline
$c_2$ & 20 & 14 & 6.15$\cdot 10^{29}$ & 6.89$\cdot 10^{26}$ & 8.29$\cdot 10^{28}$ & 10 & 3 & 6.5 \\
\hline
\end{tabular}
\end{center}

${}_{}$\\

As we can see from the degree in $\d$, both $c_1, c_2$ are not cyclic vectors.
Notably, numerous experiments in the setup of this example never led us to a cyclic vector for $A^3/A^3 M_2$. Instead we repeatedly obtained polynomials $c_i$ of degree 4. 
This shows, that there is a $4$-dimensional submodule $N$ of the $6$-dimensional module
 $A^3/A^3 M_2$, which dominates over others for the special choice of the form of test vectors. 
\end{example}

\section{Normal Form over a General Domain}
\label{NFGeneral}

As we have noted, there are many different-looking normal forms over a non-simple domain. 
\begin{definition}
Let $R$ be a ring. An element $r\in R$ is called \textbf{two-sided}, if $r$ is not a divisor of zero and ${}_{R}\langle r \rangle= \langle r \rangle_R$. It is called \textbf{proper}, if ${}_{R}\langle r \rangle \subsetneq R$.
\end{definition}

Though in \cite{Cohn} a two-sided element has been called \textbf{invariant}, we propose to use \textit{two-sided} instead, due to the ubiquity of the word \textit{invariant}.

Let $R$ be a domain and a $\K$-algebra. Then 
$r\in R$ is
proper two-sided if and only if
$\forall s,s'\in R$ $\exists t,t'\in R$ such that $rs = tr$ and $s'r = rt'$
. If $R$ admits a Gr\"obner basis theory, this is the same as to say ``$\{r\}$ is a two-sided Gr\"obner basis of the two-sided ideal $\langle r \rangle$''.
It is straightforward, that for a simple domain $0$ is the only proper two-sided element.

Generalizing the statement in the Example 4.4 of \citep{LS10} leads to the following result.
\begin{lemma}
\label{nft} 
Let $R$ be a non-simple Euclidean domain and a $\K$-algebra. 
Moreover, let $m \leq n$ be natural numbers and $r\in R$ be a proper
two-sided element. Then for the $2\times 2$ matrices $D_1 = \Diag(r^m,r^n)$ and $D_2 = \Diag(1,r^{m+n})$ the corresponding modules $M_1 = R^2/R^2 D_1$ and $M_2 = R^2/R^2 D_2$ are not isomorphic.
\end{lemma}
\begin{proof}
At first, we note that the set of proper two-sided elements is multiplicatively closed, thus $r^n,r^m, r^{n+m}$ are proper two-sided.
As we can see, $\Ann_R M_2 = \langle r^{m+n} \rangle$ and $\Ann_R M_1 = \Ann_R (R/\langle r^{m} \rangle \oplus R/ \langle r^{n} \rangle) = \langle r^{m} \rangle \cap \langle r^{n} \rangle  = \langle r^{n} \rangle$. We have to show that 
the two-sided ideals  $\Ann_R M_1, \Ann_R M_2$ are not equal, then the claim follows. Let $S = R\otimes_K R^{opp}$ be the enveloping algebra of $R$ and $\iota:R\to S$ a natural embedding of $\K$-algebras, then any two-sided ideal of $R$ is a left ideal over $S$. Suppose that for two-sided ideals $\langle r^n \rangle \subseteq \langle r^{m+n} \rangle$. Over $S$ we have the inequality of left ideals $\langle \iota(r^{n}) \rangle \subseteq \langle \iota(r^{m+n}) \rangle$, hence there exists $s\in S$, such that $\iota(r^n) = s \iota(r^{m+n}) = s \iota(r^{m}) \iota(r^{n})$. Then $(s \iota(r^m) - 1)\cdot \iota(r^n)  = 0$ and hence $s \iota(r^m) = 1$, since $S$ is a domain. Thus, $\langle \iota(r^m) \rangle = S$ as a left ideal and hence
$\langle r^m \rangle = R$, what is a contradiction to the assumption that $r$ is a proper two-sided element.
\end{proof}

Recall that for $a,b \in R$, one says that $a$ totally divides $b$ (and denotes $a \mid\mid b$) if and only if there exists a two-sided $c\in R$, such that $a \mid c \mid b$\footnote{Recall, that $a\mid c$ if $\exists d\in R$ such that either $c=ad$ or $c=da$ holds.} \citep{Cohn}. Moreover, $a \mid\mid a$ if and only if $a$ is two-sided element.

\begin{example} 
\label{exNFT}
In the first shift algebra $R_*=S_1$, cf. Example \ref{runningEx3}, consider 
the diagonal matrices
\[
D_1 = 
\left[\begin{array}{cc}
s+x & 0 \\
0 & s(s+x) 
\end{array}\right]
\quad \text{and} \quad
D_2 = 
\left[\begin{array}{cc}
s & 0 \\
0 & s(s+x) 
\end{array}\right].
\]
At first, let us analyze the appearing elements for the proper two-sidedness.
By computing two-sided Gr\"obner bases over $S_1$, we obtain that $s$ is
already such a basis, whereas ${}_{S_1}\langle s+x \rangle_{S_1} = {}_{S_1}\langle x,s \rangle_{S_1}$ and 
${}_{S_1}\langle s(s+x) \rangle_{S_1} = {}_{S_1}\langle (x+1)s,s^2 \rangle_{S_1}$. Thus neither $s+x$ nor $s(s+x)$ are proper two-sided in $R_*$. In the localization $R = (\K[x]\setminus\{0\})^{-1} R_*$, we even see that 
${}_{R}\langle s+x \rangle_{R} = R$ and
${}_{R}\langle s(s+x) \rangle_{R} = {}_{R}\langle s \rangle_{R}$.

In the matrix $D_2$ there is a total divisibility on the diagonal. Indeed, $s$ is proper two-sided and $s \mid s(s+x)$. Hence, for any $f \in R$ the result of 
reduction of $s$ with $fs(s+x)=f(s+x+1)s$ will be $(1 - f(s+x+1))s \in \langle s \rangle$. Hence, no reduction procedure will lead to
a unit instead of $s$. 

On the other hand, we see that  $s(s+x) \in \langle s \rangle$. Then $s(s+x) - fs = (s+x+1 -f)s$ has degree $1$ in $s$ only for $f=s-g(x)$, where $g(x)\neq -(x+1)$. This reduction produces $[0, s(s+x)]^T - (s-g(x))\cdot [s,s]^T = [-(s-g(x))s, (g(x)+x+1)s]^T$, which does not lower the degree in the column but only interchanges terms of degrees 1 and 2. Hence no further essential simplification is possible, since replacing $s(s+x)$ by a similar factor can neither lower the degree nor lead to a two-sided element.

In the matrix $D_1$, 
there is no total divisibility on the diagonal, hence we will be able to achieve the lower degree. And indeed, for 
\[
U = \left[
\begin{array}{*{2}{c}}
1 & 1 \\
s^{2}+xs+2s & s^{2}+xs+2s+x+2
\end{array}
\right], 
\;
V=
\left[
\begin{array}{*{2}{c}}
s+1 & -s^{2}-xs-s \\
-1 & s+x-1
\end{array}
\right]
\]
which are $R$-unimodular, we obtain that
\[
U D_1 V 
=
\left[
\begin{array}{*{2}{c}}
x & 0 \\
0 & x+2
\end{array}
\right]
\cdot 
\left[
\begin{array}{*{2}{c}}
1 & 0 \\
0 & s(s+x)(s+x-1)
\end{array}
\right].
\]

Thus $D_1$ possesses even the Jacobson normal form $\Diag(1,s(s+x)(s+x-1))$ over $R$.
\end{example}

\section{Implementation and Examples}
\label{Impl}

Our implementation of the computation of a diagonal form together with
transformation matrices is called \texttt{jacobson.lib}. It is distributed together with \textsc{Singular} \citep{Singular} since its version 3-1-0.
In the Appendix we put an example of a \textsc{Singular} session with
input, output and explanations.

\subsection{Comparison}

There are other packages to compute diagonal and Jacobson forms. 
The package \textsc{Janet} for \textsc{Maple}  \citep{Robertz,CQR}
directly follows the classical algorithm with no special optimizations.
In \textsc{Maple} packages by H.~Cheng et~al. \citep{BCL06, CL07, DCL08} modular (\textsc{Modreduce}) and fraction-free (\textsc{FFreduce}) versions of an order basis of a polynomial matrix $M$ from an Ore algebra $A$ are implemented. 
The computation of the left nullspace of $M$ and indirectly the Popov form of $M$ uses order bases.
There are also experimental implementations, mentioned in \citep{CulianezQuadrat} and \citep{Middeke}. 

In \citep{LS10}, we compared our implementation with the one by D.~Robertz.
In turned out that with our implementation, one experiences moderate swell of coefficients and obtains uncomplicated transformation matrices.

Unlike in theoretical considerations, our implementation usually returns diagonal matrices with elements of descending degrees on the main diagonal.

\subsection{Examples}
\label{Examplez}

\begin{example}
Consider two versions of the Example 3.8 in \citep{LS10}.

\textbf{(A)}. In the first Weyl algebra $R_*=K[x][\partial; \id, \frac{d}{dx}]$, consider the matrix
\begin{displaymath}
 M=\left[ \begin{array}{cc} \partial^2-1 & \partial+1 \\ \partial^2+1 & \partial-x  \end{array}\right]
\in R^{2 \times 2}.
\end{displaymath}

Then the algorithm computes
\[
UMV=
\begin{footnotesize}
\left[\begin{array}{cc}
(x+1)^2\partial^2+2(x+1)\partial-(x^2+1) & 0\\0 & 1\end{array}\right] = D,
\end{footnotesize}
\]
\[
U = 
\begin{footnotesize}
\left[ \begin{array}{cc} -(x+1)\partial+x^2+x+1 & (x+1)\partial+x \\ \partial-x &  -\partial-1\end{array}\right], 
\end{footnotesize}
\
V = 
\begin{footnotesize}
\left[\begin{array}{cc} 
1 & 0\\(x+1)\partial^2+2\partial-x+1 & 1\end{array}
\right]
\end{footnotesize}.
\]

Let us analyze the transformation matrices for $R_*$-unimodularity. Indeed,
$V$ is unimodular over $R_*$ since it admits an inverse $V'$. However, $U$ is unimodular over $S^{-1} R_*$ for $S = \{(x+1)^i \mid i\in\N_0\}$ (cf. Lemma \ref{OreWeyl}) since $U\cdot Z = W$ and $W$ is invertible in the ring containing the inverse of $(x+1)^2$, that is over a ring, containing $S^{-1}R_*$.
\[
\begin{footnotesize}
V' = \left[
\begin{array}{cc}
1 & 0 \\
-(x+1) \d^{2}-2\d+x-1 & 1
\end{array}\right], 
Z = \left[
\begin{array}{cc}
2(\d+1) & (x+1)\d+x-2 \\
2(\d-x) & (x+1)\d-x^{2}-x-3
\end{array} \right],
W = \left[
\begin{array}{cc}
0 & -4(x+1)^{2} \\
2 & 5(x+1)
\end{array}\right]
\end{footnotesize}.
\]

\textbf{(B)}.
In the first shift algebra $S_1$, cf. Example \ref{runningEx3}, one has 
\[
\Diag \begin{footnotesize}  ( \ (x+1)(x+2) s^{2}+2(x+1)s-(x-1)(x+2), 1) \end{footnotesize} = 
\]
\[
\begin{footnotesize}
\left[\begin{array}{cc}
-(x+1)s+x(x+2) & (x+1)s+x+2 \\
-s+x+1 & s+1
\end{array}\right]
\cdot
\left[\begin{array}{cc}
s^{2}-1 & s+1 \\
s^{2}+1 & s-x
\end{array}\right]
\cdot
\left[\begin{array}{cc}
1 & 0 \\
-(x+2)s^{2}-2s+x & 1
\end{array}\right]
\end{footnotesize}.
\]
Denote the equality above as $D=UMV$, it turns out, that
$V$
is even $R_*$-unimodular. $U$ becomes
invertible in a localization, containing $(x+2)^{-1}$, thus, by Lemma \ref{OreShift}, containing a ring $S^{-1}R_*$ for 
$S = \{ (x\pm n)^m \mid n,m\in\N_0\}$.
\end{example}

\begin{example}
Consider now the same matrix as in Example \ref{runningEx3} over the first
Weyl algebra $R_* = \K\langle x,\d \mid \d x = x \d +1 \rangle$, respectively
$R = \K(x)\langle \d \mid \d x = x \d +1 \rangle$:
\[
M = 
\begin{footnotesize}
\left[
\begin{array}{*{5}{c}}
(x-1)\d+x^{2}-x &  & x\d+x^{2} &  & (x+2)\d+x^{2}+2x \\
\d+x & & 0 &  & \d
\end{array}
\right]
\end{footnotesize}.
\]
We obtain
\[
D = 
\begin{footnotesize}
\left[
\begin{array}{*{3}{c}}
0 & 3x^{2}(x+2)^2 & 0 \\
0 & 0 & 1
\end{array}
\right]
\end{footnotesize}
, \
U = 
\begin{footnotesize}
\frac{1}{2}
\left[
\begin{array}{*{2}{c}}
x(x+2) \d-2(x+1) & -x(x+2)^{2}\d-(x^{4}+4x^{3}+3x^{2}-4x-4) \\
(x+1)\d-2 & -(x+1)(x+2)\d-(x^{3}+3x^{2}+x-3)
\end{array}
\right]
\end{footnotesize}
,
\]
\[
 V = 
\left[
\begin{footnotesize}
\begin{array}{*{3}{c}}
-(x^{4}-9x^{2})\d^{2}-(x^{5}-x^{3}-18x)\d-(6x^{4}-24x^{2}+18) & v_{12} & 0 \\
-(3x^{3}-27x)\d^{2}-(x^{5}+5x^{4}-9x^{3}-18x^{2}-81)\d-(x^{6}+2x^{5}-6x^{4}+8x^{3}+9x^{2}-54x) & v_{22} & 0 \\
(x^{4}-9x^{2})\d^{2}+(2x^{5}-10x^{3}-18x)\d+(x^{6}-15x^{2}+18) & v_{23} & 1
\end{array}
\end{footnotesize}
\right],
\]
where 
\begin{footnotesize}
$v_{12}= (x^{3}+3x^{2}-2x)\d^{2}+(x^{4}+3x^{3}+4x^{2}+6x+2)\d+(5x^{3}+12x^{2}-6)$, $v_{22}= (3x^{2}+9x-6)\cdot d^{2}+(x^{4}+8x^{3}+13x^{2}+11x+27)\cdot d+(x^{5}+5x^{4}+6x^{3}+14x^{2}+42x+23)$ 
\end{footnotesize}
and
\begin{footnotesize}
$v_{23}= -(x^{3}+3x^{2}-2x)\d^{2}-(2x^{4}+6x^{3}+2x^{2}+6x+2)\d-(x^{5}+3x^{4}+8x^{3}+18x^{2}+12x-6)$
\end{footnotesize}.

\noindent
$U$ is unimodular in the localization with respect to $S_U$, which is generated as monoid by $x,x+1,x+2$.
$V$ is unimodular in the localization with respect to $S_V$, which is generated by $x,x+2,x+3,x-3$. Thus, the isomorphism of modules lifts from $R$ to $S^{-1}R_*$, where $S = \{ x^{i_1}(x+1)^{i_2}(x+2)^{i_3}(x+3)^{i_4}(x-3)^{i_5} \mid i_j\in\N_0\}$.\\

The same system over the shift algebra was computed in detail in Example \ref{runningEx3}. There it turns out, that $U$ is unimodular over a ring
containing $(x+2)^{-1}$ and $V$ is unimodular over a ring 
containing the inverses of $\{x-2,x-1,x,x+3,x+4\}$. Due to Lemma \ref{OreShift}, the needed Ore set for the localization is not smaller than $S = S_U = S_V = \{ (x\pm n)^m \mid n,m \in\N_0\}$.

\end{example}


\section{Conclusion, Further Research and Open Problems}

\subsection{Further Research}

We do not perform an analysis of the theoretical complexity of our algorithm, since it has been designed by using Gr\"obner bases, thus the formal complexity is too high. Indeed, one can see, that we use Gr\"obner bases rather for convenience, that is in order to present an algorithm, working over arbitrary OLGAED. The original (non-fraction-free) algorithm can be reformulated in terms of extended left and right greatest common divisors over a concrete algebra. To the best of our knowledge, the theoretical complexity of such GCD algorithms depends on the given algebra. For rational Weyl algebras over general fields of characteristic zero there is a paper \citep{Grig90}. It is interesting to find estimations for classical operator algebras like the shift algebra, $q$-Weyl and $q$-commutative algebras with rational coefficients. Having this information, one could perform complexity analysis for the kind of algorithms we propose. It was reported in \citep{Middeke}, that the Jacobson form can be computed in polynomial time. In our opinion this should be formulated more precisely in the framework, proposed in \citep{Grig90}. We do not claim that our algorithm is superior in terms of theoretical complexity to the others but stress, that its fraction-free version is widely applicable in practical computations. In particular, our algorithm returns transformation matrices with reasonable sizes of their entries and their coefficients. 

Concerning the further development of our implementation in \textsc{Singular} \citep{Singular} called \texttt{jacobson.lib}, we plan the following enhancements.
In order to provide diagonal form computations over $q$-algebras like $q$-commutative, $q$-shift, $q$-difference and $q$-Weyl algebras, we need to find involutive antiautomorphisms. It turns out, that there are no $\K(q)$-linear but $\K$-linear involutions, which act on $q$ as a non-identical involutive automorphism of $\K(q)$.
We will work on finding such involutions and adapt the implementation to the more general situation. Moreover, it is possible to employ modular Gr\"obner bases in the implementation.

We have demonstrated, how working with a ring $R_*$ and its localization $S^{-1}R_*$ with respect to some Ore set $S\subset R_*$ is used in various situations. In this paper we concentrated on rings of polynomials respectively rational functions. The same ideas in theory immediately apply to operator algebras $R_* = K[[x]]\langle \d \rangle$ and $R = K((x))\langle \d \rangle$. However, in practical computations with the computer we are restricted to finite power series, presented by a Laurent polynomial.

For the Weyl and $q$-Weyl algebras over a field $\K$ with $\Char \K=0$, algorithmic computations are possible over the local ring $\K[x]_{\mathfrak{m}_p}$ for $p\in\K^n$ and $\mathfrak{m}_p = \langle x_1 - p_1,\ldots,x_n - p_n \rangle$. Notably, the Ore completion of the set $\K[x]\setminus\{\mathfrak{m}_p\}$ over shift and $q$-shift algebras will contain zero. Thus, there is no
analogon to the situation with ($q$)-Weyl algebras.


\subsection{Open Problems}

\textbf{1}. From the description and the proof of the Algorithm \ref{diagonalPoly}, we need in principle only the existence of a terminating Gr\"obner basis algorithm over $R$ resp. $R_*$. The latter relies on the constructivity of basic operations with fractions in $R$ via $R_*$. Thus the Algorithm \ref{diagonalPoly} immediately extends to the case of a general OLGA $(R,B,I)$ (cf. Def. \ref{OLGA}) for a nontrivial completely prime ideal $I\subset B$ by \cite{BGV}. In which bigger class of algebras one can perform concrete computations? One can work with local rings like $\K[x]_{\mathfrak{m}_p}$, $p\in\K^n$ and $\K[[x]]$ for $\Char \K =0$ or with $\K\{x\}$ for $\K=\R,\C$ as coefficient domains.

\textbf{2}. The algorithm computes $U,V,D$ such that $UMV=D$. It is clear,
that $U,V$ are not unique. By fixing $M$ and $D$, how can one describe the set $\{ (U,V) \mid UMV=D\}$?
Does there exist $U'$ resp. $V'$ from this set, such that special properties like unimodularity hold? This is connected, in particular, to the recent results of A.~Quadrat and T.~Cluseau \citep{CQ10}. Namely, they prove the existence of matrices $U$ and $V$, which are unimodular over $R_*$ for certain situations and present an algorithm for the computation of these matrices.

\textbf{3}. One of the most important problems in non-commutative computer algebra is to show algorithmically, that two given finitely presented modules are not isomorphic. 
It is still open even for cyclic modules over general simple Euclidean Ore domain. Namely, let $R$ be the latter domain and $a,b\in R$ with $\deg a = \deg b >0$. Up to now there is no algorithm, which determines that $R/\langle a \rangle \not\cong R/\langle b \rangle$. However, there are several situations, where 
this question has been solved.

In \cite{CQ08}, the authors presented a semi-algorithm for finitely presented modules. Given degree bounds for the variables, the algorithm first looks for matrices, determining a homomorphism. If such a pair has been found, the kernel and cokernel of the corresponding homomorphism are computed and returned.

\textbf{4}. As we have seen, over non-simple domains we have many different normal forms. Starting with classical operator algebras like the ($q$-)shift algebra and $q$-Weyl algebra, it is important
to describe possible normal forms.

\textbf{5}. As in the previous item, what is the most probable normal form for a random square matrix (which is then of full rank) over a non-simple domain? Though it seems that this might depend on the algebra, we conjecture the $\Diag(1,..,1,p)$, that is a Jacobson form, is the most probable one. Such approach might lead to the classification as in \textbf{4}.

\section*{Acknowledgments}

We express our gratitude to Eva Zerz and Hans Sch\"onemann for discussions and their advice on numerous aspects of the problems, treated in this article. We thank Daniel Robertz, Johannes Middeke and Howard Cheng for explanations about respective implementations. We are grateful to Mark Giesbrecht, George Labahn and Dima Grigoriev for discussions concerning normal forms for matrices and complexity of algorithms. Many thanks to Moulay Barkatou and Sergey Abramov for explanations of their algorithms for linear functional systems. 
The first author is grateful to the SCIEnce project (Transnational access) at RISC for supporting his visits to RISC and the usage of computational infrastructure at RISC.
We would like to thank the anonymous referees, who helped us a lot with their insightful remarks and suggestions. Kristina Schindelar carried out her work on this paper during her Ph.D. studies at Lehrstuhl D f\"ur Mathematik of RWTH Aachen University.


\section{Appendix. The Code for an Example}
\label{App}
Consider the code for Example \ref{runningEx3}. 
After starting a \textsc{Singular} session, we need to set up the algebra first.

\begin{verbatim}
> LIB "jacobson.lib"; // load the library
> ring w = 0,(x,s),(a(0,1),Dp);
\end{verbatim}
\noindent
The ring \texttt{w} is a commutative one in variables \texttt{x,s} over the
field $\Q$ (0 stands for the characteristic). The substring \texttt{(a(0,1),Dp)} defines a monomial ordering on \texttt{w}, which in this case is an
extra weight ordering, assigning weights 0 to x and 1 to s. If two monomials
are of the same weight, they are further compared with \texttt{Dp} (degree lexicographic ordering). The described ordering mimicks the ordering on the rational shift algebra. For computations one can also use other orderings, like \texttt{Dp} or \texttt{dp}, cf. the \textsc{Singular} manual.

\begin{verbatim}
> def W=nc_algebra(1,s); // set up shift algebra
> setring W; 
\end{verbatim}
\noindent
This code creates and sets active a new ring \texttt{W} as a non-commutative ring from \texttt{w} with the relation $s \cdot x = 1\cdot xs + s$.
By executing \texttt{W;} a user will obtain the description of the
active ring.

\begin{verbatim}
> matrix m[2][3]=
>   x*s-s+x^2-x, x*s+x^2, x*s+2*s+x^2+2*x,
>   s+x,0,s;
> list J=jacobson(m,0);
\end{verbatim}
\noindent
Here the matrix was entered and the algorithm called. Note, that putting
a string \texttt{printlevel=2;} before the \texttt{jacobson} call will
output the progress of the algorithm. Higher values of \texttt{printlevel}
will lead to more details printed during the execution.
The list \texttt{J} has three entries, namely $U,D,V$.

\begin{verbatim}
> print(J[1]*m*J[3]-J[2]); // check UmV=D
==> 0,0,0,
    0,0,0 
> print(J[2]); // that is D
==> 0,x4+3x3-x2-3x,0,
    0,0,           x 
> print(J[1]); // that is U
==> 0,  -1/4x2-x-3/4,
    1/4,-1/4x-1/2
> print(J[3]); // that is V
==> _[1,1],_[1,2],xs2-s2+x2s-xs,
    _[2,1],_[2,2],_[2,3],       
    _[3,1],_[3,2],_[3,3]
\end{verbatim}
\noindent
The symbols \verb?_[2,2]? before are printed, when the entries are long.
One can access them as single polynomials
\begin{verbatim}
> print(J[3][2,2]);
==> -21s2-7x2s-2xs+11s-3x3-2x2+13x-8
\end{verbatim}

\end{document}